\documentclass[11pt]{amsart}
\usepackage{palatino}
\usepackage[all]{xy}
\usepackage{amssymb}
\usepackage{latexsym}
\usepackage{amscd}
\usepackage{color}

\newtheorem{theorem}{Theorem}[section]
\newtheorem{lemma}[theorem]{Lemma}

\newtheorem{proposition}[theorem]{Proposition}
\newtheorem{corollary}[theorem]{Corollary}
\newtheorem{conjecture}[theorem]{Conjecture}
\newtheorem{remark}{Remark}

\def \<{\langle}
\def \>{\rangle}

\def \a{\alpha }

\def \l{\lambda }

\newcommand{\bea}{\begin{eqnarray}}
\newcommand{\eea}{\end{eqnarray}}
\newcommand{\be}{\begin {equation}}
\newcommand{\ee}{\end{equation}}

\newcommand{\Z}{\Bbb Z}

\newcommand{\X}{\mathfrak{X}}

\newcommand{\Zp}{{\Bbb Z}_{>0} }

\newcommand{\N}{{ \mathbb Z}_{\ge 0} }
\newcommand{\C}{\Bbb C}

\newcommand{\WW}{\boldsymbol{ \mathcal{W}}}

\newcommand{\la}{\langle}
\newcommand{\ra}{\rangle}

\newcommand{\halmos}{\rule{1ex}{1.4ex}}

\newcommand{\epfv}{\hspace*{\fill}\mbox{$\halmos$}}

\newcommand{\nn}{\nonumber \\}
\usepackage{amssymb}

\begin{document}
\title[On W-algebras associated to $(2,p)$ minimal models and their representations] {On W-algebras associated to $(2,p)$ minimal models and their representations}
\author{Dra\v{z}en Adamovi\'c and Antun Milas}
\address{Department of Mathematics, University of Zagreb, Croatia}
\email{adamovic@math.hr}

\address{Department of Mathematics and Statistics,
University at Albany (SUNY), Albany, NY 12222}
\email{amilas@math.albany.edu}

 \markboth{Dra\v zen Adamovi\' c and Antun Milas} { }
\bibliographystyle{amsalpha}

\begin{abstract}
For every odd $p \geq 3$, we investigate representation theory of the vertex algebra $\WW_{2,p}$
associated to $(2,p)$ minimal models for the Virasoro algebras. We
demonstrate that vertex algebras $\WW_{2,p}$ are $C_2$--cofinite and
irrational. Complete classification of irreducible representations for
$\WW_{2,3}$ is obtained, while the classification for $p \geq 5$ is subject to certain constant
term identities. These identities can be viewed as "logarithmic deformations" of Dyson
and Selberg constant term identities, and are of independent interest.

\end{abstract}
\maketitle

\tableofcontents

\section{Introduction}

The Virasoro algebra modules of central charge zero play a prominent rule in conformal field theory (CFT). Most recently they also appeared in the physics literature on
logarithmic CFT (cf.  \cite{GL}, \cite{FGST-log}, \cite{EF}, \cite{GRW}, \cite{MR}, etc.).
As with the ordinary CFTs,  it is desirable to understand constructions in logarithmic CFT  via representations of vertex algebras. While this can be done is some cases,
there are examples of central charge zero LCFT which are still poorly understood, at least from algebraic point of view.

In an important paper \cite{FGST-log} (cf. also \cite{FGST-qg}), certain vertex $\mathcal{W}$-algebras, denoted by $\WW_{q,p}$, associated to central charge
$c_{q,p}=1-\frac{6(p-q)^2}{pq}$ Virasoro minimal models were introduced (here $p$ and $q$ are co-prime integers). These vertex algebras are extracted from the Felder's
complex via a pair of screening operators. Similar construction, but with one screening operator, has also been used in the definition of the triplet vertex algebra
$\mathcal{W}(p)$ \cite{AdM-triplet}, a certain extension of the Virasoro vertex algebra $L(c_{1,p},0)$. Compared to the triplet $\mathcal{W}(p)$, the vertex algebra
$\WW_{q,p}$ is more difficult to study due to more complicated structure of Virasoro modules and screening operators. Another sharp contrast
is that $\WW_{q,p}$ is no longer  simple nor self-dual and in fact  combines into a nontrivial extension of the Virasoro vertex algebra $L(c_{q,p},0)$ (the latter is
known to be rational \cite{W}). More precisely,
$$0 \longrightarrow \mathcal{R}_{q,p} \longrightarrow \WW_{q,p} \longrightarrow L^{Vir}(c_{q,p},0) \longrightarrow 0,$$
where $\mathcal{R}_{q,p}$ is  the maximal vertex ideal in $\WW_{q,p}$. The reader should notice that for  $q=2$, $p=3$ (i.e. $c=0$) the above sequence is non-trivial,
even though
the vertex algebra $L^{Vir}(0,0) \cong \mathbb{C}$ is trivial.

While the paper \cite{FGST-log} does provide expected  "spectrum" and "generalized characters"
of $\WW_{q,p}$, several results in \cite{FGST-log} rely on a
conjectural correspondence between the category of  $\WW_{q,p}$-modules and the category
of modules of a certain quantum group. Thus, in parallel with the triplet algebra \cite{AdM-triplet}, it is
not clear how to:  (i)  obtain a "strong" set of generators of
$\WW_{q,p}$, (ii) prove the  $C_2$-cofiniteness of $\WW_{q,p}$ and
(iii) classify irreducible $\WW_{q,p}$-modules, and finally (iv)
construct projective covers. Some partial results in this
direction were previously obtained in \cite{FGST-log} for all $q$ and $p$, and in \cite{GRW}
for $c=0$. In addition, in \cite{AdM-2009} we found an explicit
construction of some (but not all!) logarithmic $\WW_{q,p}$-modules.

Arguably, the simplest  minimal models occur within the series $(2,p)$, $ p \geq 3$ odd. These modules have
been linked to classical partition identities (cf. \cite{FF}) and their combinatorics and characters are somewhat less complicated compared to the unitary series
$(k,k+1)$.
So, at least from this point of view, it seems natural to contemplate $\WW_{2,p}$ algebras and their representations first.

This paper is meant to provide a very detailed vertex algebraic study of $\mathcal{W}$-algebras $\WW_{2,p}$,
 following our approach in \cite{AdM-triplet}, but supplied
with many new ideas. In parallel with \cite{AdM-triplet}, we study $\WW_{2,p}$  as a subalgebra of a rank one lattice vertex algebra $V_L$ such that the conformal vector
$\omega$ has central charge $c_{2,p}$ (cf. Section \ref{konstrukcija}).  Although  we work in full generality of $\WW_{2,p}$ vertex algebras, our original motivation
was to understand the algebra $\WW_{2,3}$ (cf. \cite{GRW}, \cite{EF}, \cite{FGST-log}, etc).
A central aim of this work is to study the $C_2$-cofiniteness of the vertex operator algebra $\WW_{2,p}$. This is an important problem
because $C_2$-cofiniteness implies that $\WW_{2,p}$ admits only
finitely many irreducible representations. That is our first result
\begin{theorem} \label{c2-uvod}
For any $p \geq 3$, the vertex operator algebra $\WW_{2,p}$  is $C_2$-cofinite.
\end{theorem}
We should mention that for $p=3$, Theorem \ref{c2-uvod} was (at least implicitly) conjectured in \cite{GRW}.  The $C_2$-cofiniteness is also important
because it allows us to use powerful results from \cite{HLZ} and \cite{Hu1}. In particular,  results there
endow the category of $\WW_{2,p}$-modules with a natural braided vertex tensor category structure.
The fusion rules for
this category are of considerable interest in the physics literature (cf. \cite{GRW}, \cite{Wo}).

A few words about the proof of Theorem \ref{c2-uvod}. As in \cite{A-2003}, \cite{AdM-2007} (see also \cite{AdM-striplet})  we  study the singlet vertex algebra
$\overline{M(1)}$ strongly generated by two generators and realized as a subalgebra of $\WW_{2,p}$. In this way we realize  a new  family of $\mathcal{W}$-algebras with
two generators. We  also completely determine  the associated Zhu's algebras. It turns out that the structural results on  singlet vertex algebra $\overline{M(1)}$
provide an  important step in the proof of Theorem \ref{c2-uvod}.

Having established the $C_2$-cofiniteness we can now go on and explore irreducible $\WW_{2,p}$-modules.
But to achieve this in full generality we have to overcome some difficulties of combinatorial nature (see Section \ref{klasifikacija-gen}).
First we only state and prove classification theorem for $p=3$.
\begin{theorem}
The vertex operator algebra $\WW_{2,3}$ has precisely $13$ irreducible inequivalent modules (explicitly described in the paper). Each irreducible module
is a submodule  or a subquotient of an irreducible module for the rank one lattice vertex algebra $V_L$ of central charge zero.
\end{theorem}
The classification of modules for $\WW_{2,p}$ is subject to a constant term identity closely related to classical Dyson-Selberg's identities \cite{An}, but with
important modifications due to logarithmic factors.  We have
\begin{theorem} \label{class-intro}
If Conjecture \ref{conj-g}  holds, the vertex algebra $\WW_{2,p}$ has precisely $4p+\frac{p-1}{2}$ irreducible modules.
\end{theorem}

As with the triplet vertex algebra it is important to explore the space of generalized characters (cf. \cite{Miy}) for purposes of  modular invariance.
We have a partial result in this direction.
\begin{theorem}
The space of generalized $\WW_{2,p}$-characters is $20$-dimensional for $p=3$. For $p \geq 5$, this space is conjecturally of dimension $\frac{15p-5}{2}$.
\end{theorem}

There are several directions we plan to pursue at this stage.
Many ideas and arguments presented here certainly generalize to other $\WW_{q,p}$-algebras and possibly to higher $\mathcal{W}$-algebras \cite{Ar}. But even before diving
deeper into $\WW_{q,p}$
algebras, it would be desirable to construct indecomposable $\WW_{2,p}$ algebras. For instance, the methods in \cite{AdM-2009}, \cite{FFHST} and this paper can be
employed to study logarithmic $\WW_{2,p}$-modules. Another direction is to explore possible connections between our construction and generalized twisted modules
introduced
recently in \cite{Hu2}.

\vskip 2mm

\noindent {\bf Acknowledgments:}
We would like to thank Yi-Zhi Huang for useful discussion.

\section{A construction of derivations of vertex superalgebras}
\label{der-construction}

In this section we shall present a new
construction of derivation on a vertex subalgebra realized as the
kernel of a  screening operator. This construction will be
applied to lattice vertex algebras. Further applications will be
studied in our future work.

\vskip 5mm

We assume some familiarity with vertex operator algebras and superalgebras. We refer the reader to
standard textbooks on this subject (cf. \cite{LL} for instance).
Assume that $(V,Y,{\bf 1}, \omega)$ is a $
\mathbb{Z}$--graded vertex operator superalgebra
with
$V=V_0 \oplus V_1$ (parity decomposition)
and the vertex
operator map $Y( \cdot, x)$, where $Y(u,x)=\sum_{n \in \mathbb{Z}}
u_n x^{-n-1}$ and $Y(\omega,x)=\sum_{m \in \mathbb{Z}}
L(m)x^{-m-2}$. Let $a \in V$ be odd vector  such that
\bea
&& \{ a_n, a_m \} =  a_n a _m + a_m a _n =0 \quad \forall n,m \in {\Z}, \label{rel-c-1} \\
&& L(n) a = \delta_{n,0} a \quad \forall \ n \in {\N},
\label{rel-c-2}
\eea
so that $a$ is of conformal weight one. Then $$[a_0, L(n)]=0,$$
i.e., $a_0$  is a screening operator. Moreover, $$\overline{V} =
\mbox{Ker} _V a _0$$ is a vertex subalgebra of $V$. We shall now
construct a derivation on $\overline{V}$. Recall that an (even) derivation on a superalgebra  $V$ is
a linear map $D: V \rightarrow V$, such that
$$D(a_n b)=(Da)_n b+ a_n(Db),$$
for all $a,b \in V$ and $n \in \mathbb{Z}$. Define the following operator
$$ G = \sum_{i = 1 } ^{\infty} \frac{1}{i} a _{-i} a
_i.$$
Since $G$ commutes with the action of screening operator $a_0$, we
have that $G$ is a well-defined operator on $\overline{V}$.
\begin{theorem} \label{G-der-general}
On the vertex superalgebra  $\overline{V}$ we have:
\item[(i)] $[L(n), G] = 0$ for every $n \in {\Z}$, i.e, $G$ is a
screening operator which commutes with the action of the Virasoro
algebra.

\item[(ii)] The operator $G$ is a derivation on the vertex
algebra $\overline{V}$, and in particular, $\exp [  G]$ is an
automorphism of $\overline{V}$.

\item[(iii)] Assume that $(W,Y_W)$ is any $V$--module and $ v \in \overline{V}$. Then on the submodule $\mbox{Ker}_{W}
a_0 $ we have
$$ Y_W ( G v, z) = [G, Y_W (v,z)].$$
\end{theorem}
 \begin{proof}(i) We start with the formula
$$[L(n),Y(a,x)]=(x^{n+1}\frac{d}{dx}+(n+1)x^{n})Y(a,x).$$
Thus,
$$[L(n),a_m]=-m a_{m+n}$$
and hence
$$\sum_{m=1}^\infty \frac{1}{m} [L(n),a_{-m} a_m]=\sum_{m=1}^\infty a_{-m+n} a_m
-\sum_{m=1}^\infty a_{-m} a_{m+n}.$$ But the last expression is zero
because $a_k$ and $a_l$ anti-commute for any $k$ and $l$, and because $a_0$
acts trivially on $\overline{V}$.

For (ii) we first observe \be \label{decomp}
[G,Y(u,x)]=\sum_{m=1}^\infty  \frac{1}{m} a_{-m}
[a_m,Y(u,x)]+\sum_{m=1}^\infty \frac{1}{m} [a_{-m},Y(u,x)]a_m. \ee
Let us denote the first and second summand on the right hand side in
(\ref{decomp}) by $A$ and $B$, respectively. Then
\bea
&& A=\sum_{n=1}^\infty   \sum_{m=n}^\infty x^{m-n} \frac{a_{-m}}{m} {m \choose n}
 Y(a_n u,x)+\sum_{n=1}^\infty \frac{a_{-n} x^n}{n}Y(a_0 u,x) \nonumber \\
&& =\sum_{n=1}^\infty \sum_{m = n}^\infty x^{m-n} \frac{a_{-m}}{m} {m \choose n} Y(a_n u,x), \nonumber
\eea
where we used the property $a_0|_{\overline{V}}=0$. Now replace
$m-n$ by $m$, so we get
$$A=\sum_{n=1}^\infty  \sum_{m=0}^\infty  x^m \frac{a_{-n-m}}{m+n} {m+n \choose n} Y(a_n u,x)$$
\be \label{comm} =\sum_{n=1}^\infty  \sum_{m=0}^\infty x^m
\frac{a_{-n-m}}{n} {m+n-1 \choose m} Y(a_n u,x). \ee

Similarly,
\bea
&& B=\sum_{m=1}^\infty \frac{1}{m} [a_{-m},Y(u,x)]a_m \nn
&& =\sum_{m=1}^\infty \sum_{n=0}^\infty  \frac{1}{m} {-m \choose n} x^{-m-n} Y(a_n u,x)a_m \nn
&& =\sum_{m=1}^\infty \sum_{n=0}^\infty  \frac{1}{m}  (-1)^n {m+n-1 \choose n} x^{-m-n} Y(a_n u,x)a_m. \nonumber
\eea

On the other hand, the iterate formula (see \cite{LL}) gives for odd
$w$ (for even $w$ just switch the sign):

$$Y(a_{-m}w,x_2)={\rm Res}_{x_1} \left( \frac{1}{(x_1-x_2)^m}
 Y(a,x_1)Y(w,x_2)+\frac{1}{(-x_2+x_1)^m} Y(w,x_2)Y(a,x_1) \right)$$
$$=\sum_{i = 0}^\infty a_{-m-i} {-m \choose i}(-1)^i x_2^{i} Y(w,x_2)$$
$$+\sum_{i = 0}^\infty  {-m \choose i}(-1)^{m-i} x_2^{-m-i} Y(w,x_2) a_i$$
$$=\sum_{i = 0}^\infty  a_{-m-i} {m+i-1 \choose i} x_2^{i} Y(w,x_2)$$
$$+\sum_{i = 0}^\infty  {m+i-1 \choose i}(-1)^{m} x_2^{-m-i} Y(w,x_2) a_i.$$
Now,

$$\sum_{m=1}^\infty \frac{1}{m} Y(a_{-m} a_m w,x_2)=$$
\be \label{iter} \sum_{m=1}^\infty \sum_{i = 0}^\infty  \frac{1}{m}
a_{-m-i} {m+i-1 \choose i} x_2^{i} Y(a_m w,x_2)+ \sum_{m=1}^\infty
\sum_{i = 0}^\infty  \frac{1}{m}   {m+i-1 \choose i}(-1)^{m}
x_2^{-m-i} Y(a _m w,x_2)  a_i. \ee Denote the first and second sum on
the right hand side in (\ref{iter}) by $A'$ and $B'$, respectively.

It is clear that $A=A'$ (just apply $m \rightarrow i$ and  $n
\rightarrow m$ in formula (\ref{comm})). Similarly one shows $B=B'$.

It is clear that all above formulas and calculations hold on
$\mbox{Ker}_W a_0$, where $W$ is any $V$-module. The proof follows.
\end{proof}

Assume now that $h \in V$ is an even vector  such that

\bea
&& L(n) h = \delta_{n,0} h, \quad n \in {\N} \label{rel-h-1}, \\
&& [h(n), h(m) ] =  \gamma \delta_{n+m,0}, \quad  \gamma \in
\mathbf{Q} \label{rel-h-2}, \\
&& h(0)  \ \mbox{acts semisimply on} \ V \    \mbox{with
eigenvalues in} \ \tfrac{1}{2} \mathbb{Z} \label{rel-h-3}, \\
&& h(n) a = \tfrac{1}{2} \delta_{n,0} a, \quad n \in {\N}
\label{rel-h-4} . \eea

Then
$ \sigma= \exp [ 2 \pi i h(0) ] $ is an automorphism  of $V$ of
order two.

Let $(W,Y_W)$ be any $V$--module. Then we can construct the
$\sigma$-twisted $V$--module $(W^{\sigma}, Y_W ^{\sigma})$ as
follows (cf. \cite{Li-tw}):
\bea
&& W^{\sigma}:= W \quad \mbox{as vector space}, \nonumber \\
&& Y_W^{\sigma} (\cdot,x) := Y_W( \Delta(h,x) \cdot,x), \ \
\mbox{where}\nonumber
\\
&&  \Delta(h,x) :=x ^{h(0)} \exp \left( \sum_{n = 1} ^{\infty}
\frac{h(n)}{-n} (-x) ^{-n} \right). \nonumber
\eea

The $\Delta$-operator satisfies \be \label{conjugate-delta}
\Delta(h,x_2)Y_W(u,x_0)\Delta(h,x_2)^{-1}=Y_W(\Delta(h,x_2+x_0)u,x_0).
\ee In particular for $u=a$ we get \be \label{conjugate-delta-1}
\Delta(h,x_2)Y_W(a,x_0)\Delta(h,x_2)^{-1}=Y_W((x_2+x_0)^{1/2} a ,x_0).
\ee

Let
$$Y_W ^\sigma(a,x)=\sum_{m \in \mathbb{Z}} a _{m+1/2} x^{-m-1/2}$$
and
$$G^{tw}=\sum_{m=0}^\infty \frac{1}{m+1/2} a_{-m-1/2} a_{m+1/2}.$$

Then we have a $\sigma$-twisted version of Theorem
\ref{G-der-general}.

\begin{theorem} \label{Gtw-der-gen}
On any $\sigma$--twisted $V$--module $ (W^{\sigma}, Y_W ^{\sigma})$,
we have:
\begin{itemize}
\item[(i)] $[L(n),G^{tw}]=0 \ \ {\rm for}  \ \ n \in \mathbb{Z}$,
i.e., $G ^{tw}$ is a screening operator.

\item[(ii)] For any $v \in \overline{V}$ we have
$$[G^{tw},Y_W ^{\sigma}(v,x)]=Y_W ^{\sigma}(G v, x). $$
\end{itemize}
\end{theorem}
\begin{proof} The strategy of the proof is as in Theorem \ref{G-der-general}, so we
omit unnecessary details. From the twisted commutator formula (cf.
\cite{Li-tw}) we get \bea \label{comm-tw}
&& [G^{tw},Y_W^\sigma(v,x)] \nonumber \\
&& =\sum_{m=0}^\infty \frac{1}{m+1/2} a_{-m-1/2}[a_{m+1/2},Y_W
^\sigma(v,x)]+
\sum_{m=0}^\infty \frac{1}{m+1/2}[a_{-m-1/2},Y_W ^\sigma(v,x)]a_{m+1/2} \nonumber \\
&& =\sum_{m=0}^\infty \sum_{n=0}^\infty \frac{1}{m+1/2} x^{m+1/2-n} {m+1/2 \choose n} a_{-m-1/2} Y_W ^\sigma(a_n v,x) \nonumber \\
&& + \sum_{m=0}^\infty \sum_{n=0}^\infty \frac{1}{m+1/2} x^{-m-1/2-n}
{-m-1/2 \choose n} Y^\sigma_W(a_n v,x)a_{m+1/2} \eea

On the other hand, the formula (\ref{conjugate-delta-1}) and the
(untwisted) iterate formula gives \bea \label{iterate-tw}
&& Y_W^\sigma(G v ,x)=Y_W(\Delta(h,x)G v,x)=\sum_{m=1}^\infty  \frac{1}{m}  Y_W(\Delta(h,x) a_{-m} a_m a,x) \nonumber \\
&& =\sum_{m=1}^\infty  \sum_{n=0}^\infty \frac{1}{m} { 1/2 \choose n} x^{1/2-n} Y_W(a_{n-m} \Delta(h,x) a_m v,x) \nonumber \\
&& = \sum_{m=1}^\infty  \sum_{n=0}^\infty \sum_{j=0}^\infty
\frac{1}{m}  (-1)^j {1/2 \choose n} {n-m \choose j} x^{1/2-n+j }
a_{n-m-j-1/2}
Y_W ^\sigma(a_m v,x)+ \nonumber \\
&& + \sum_{m=1}^\infty  \sum_{n=0}^\infty \sum_{j=0}^\infty
\frac{1}{m}   (-1)^{m+n+j} {1/2 \choose n} {n-m \choose j}
x^{1/2-j-m  } Y_W^\sigma(a_m v,x) e^\alpha_{j-1/2}, \eea where as in
the previous theorem we assumed that $v$ is odd.

The rest is an application of Vandermonde convolution for binomial
coefficients appearing in (\ref{iterate-tw}) and matching
appropriate terms in (\ref{comm-tw}) and (\ref{iterate-tw}).  In
the process we also use the fact that $a_0$ acts trivially on
$\overline{V}$.
\end{proof}

\section{Lattice vertex algebras, $\WW$-subalgebras and screening operators}
\label{konstrukcija}

Assume that $p$ is an odd natural number, $p \ge 3$. Let $$L = {\Z}
\alpha, \ \ \ \la \alpha , \alpha \ra = p.$$ Also, let $$V_L=M(1) \otimes \mathbb{C}[L]$$ be the associated
vertex superalgebra, where as usual $M(1) \cong U(\hat{\mathfrak{h}}_{<0})$ is the vacuum module for the
corresponding Heisenberg algebra $\hat{\mathfrak{h}}$, $\mathfrak{h}=L \otimes_{\mathbb{Z}} \mathbb{C}$.
Consider $D= {\Z} (2 \a)$. Then $V_D$ is a vertex subalgebra of $V_L$.

Define the Virasoro vector

$$\omega = \frac{1}{2 p} (\a (-1) ^2 + (p-2) \a (-2))$$
and screening operators
$$Q = e ^{\a}_0, \quad \widetilde{Q} =e ^{ - \tfrac{2}{p}\a \ }_0.$$
Define the following vertex (super)algebras
$$\overline{V_L}=\mbox{Ker}_{V_L}  Q \cap \mbox{Ker}_{V_L}
\widetilde{Q} $$ and
$$ \WW_{2,p} = \mbox{Ker}_{V_D}  Q \cap \mbox{Ker}_{V_D}
\widetilde{Q}.$$
Then $\overline{V_L}$ is a ${\N}$--graded vertex superalgebra and $
\WW_{2,p} \subset \overline{V_L} $ is a vertex operator algebra of
central charge $c_{2,p}=1-\frac{3(p-2)^2}{p}$.

In particular, if $p=3$ the central charge is zero; this case is of considerable interest in the paper.

Now we shall apply the general construction from Section
\ref{der-construction} in the case of lattice vertex superalgebra
$V_L$. Then we have  the following operator
$$ G = \sum_{i = 1 } ^{\infty} \frac{1}{i} e ^{\a}_{-i} e ^{\a}
_i.$$
It is convenient to express $G$ as
\be \label{G-log}
- {\rm Res}_{x_1,x_2} {\rm ln}(1-\frac{x_2}{x_1}) Y(e^{\alpha},x_1)Y(e^{\alpha},x_2).
\ee

Then applying Theorem \ref{G-der-general} we get:

\begin{theorem} \label{G-der}
On the vertex algebras $\WW_{2,p}$ and $\overline{V_L}$ we have:
\item[(i)] $[L(n), G] = 0$ for every $n \in {\Z}$, i.e, $G$ is a
screening operator.

\item[(ii)] $[Q,G] = [ \widetilde{Q}, G] = 0$.

\item[(iii)] The operator $G$ is a derivation on the vertex
algebra ${\rm Ker}_{V_L}(Q)$, and in particular on
$\overline{V_L}$ and $\WW_{2,p}$.

\item[(iv)] Relations (i)-(iii)  hold on $$\mbox{Ker}_{ \ V_{L+ \tfrac{\ell}{p} \a}} Q,$$
where $\ell \in \{0, \dots, p-1\}$.
\end{theorem}

Next we shall consider twisted $V_L$--modules. As in
\cite{AdM-sigma}, we have that
$$ \sigma :=\exp[ \frac{\pi i}{p} \a (0) ]. $$
is a canonical automorphism of order two of the vertex
superalgebra $V_L$. Moreover,  $V_L$ admits $\sigma$-twisted
modules whose description we briefly recall here \cite{Li-tw}.

Any irreducible $\sigma$-twisted $V_L$ module $M$ is isomorphic to
$$(V_{L+ \frac{(\ell + 1/2) \alpha}{p}},Y^\sigma) \ \quad (0 \le \ell \le p-1),$$ where $$V_{L+
\frac{(\ell + 1/2) \alpha}{p}}= V_{L+ \frac{\ell+  \alpha}{p}} \quad
\mbox{ as a vector spaces},$$ and the  twisted vertex operator
$Y^\sigma$ is defined
$$Y^\sigma( \cdot,x)=Y(\Delta(\alpha / 2p, x) \cdot,x).$$

Then  applying Theorem \ref{Gtw-der-gen}  in the case
$$ V = V_L, \ a = e ^{\alpha}, \ h = \frac{\alpha}{2 p}, $$
we get  the following result:

\begin{proposition} \label{Gtw-der}
\begin{itemize}
\item[(i)] The relation $$[L(n),G^{tw}]=0 \ \ {\rm for}  \ \ n \in
\mathbb{Z}$$ holds on any $\sigma$-twisted $V_L$-module.
\item[(ii)]The operator $G^{tw}$ is a $\sigma$-twisted derivative
of ${\rm Ker}_{V_L} Q$, and thus of $\WW_{2,p}$. In other words,
for any $ v \in {\rm Ker}_{V_L} Q$ we have
$$[G^{tw},v_n]=(G v)_n \ \ {\rm for} \  \ n \in \tfrac{1}{2} {\Z}. $$
\end{itemize}
\end{proposition}

\section{Fusion rules for certain Virasoro modules of central  charge $c_{2,p}$}

Let us fix some notation. In what follows $V^{Vir}(c,0)$
denote the universal Virasoro vertex operator algebra of central charge $c$ (cf. \cite{LL} for details).
The corresponding simple vertex algebra will be denoted by $L^{Vir}(c,0)$. It is known (cf. \cite{LL}, \cite{W})
that any highest weight irreducible module $L^{Vir}(c,h)$, $h \in \mathbb{C}$ is a $V^{Vir}(c,0)$-module.
We remark that for $c=0$ the algebra $L(0,0)$ is $1$-dimensional.

The aim of this section is to analyze fusion rules among  certain
irreducible Virasoro modules of central charge $c_{2,p}$, in particular $c=0$, viewed as
modules for the vertex algebra $V^{Vir}(c_{2,p},0)$.

Let
$$h_{r,s}=\frac{(pr-2s)^2-(p-2)^2}{8p}.$$
Here we analyze fusion rules  among $L^{Vir}(c_{2,p},3p-2)$ and some
special modules of type  $L^{Vir}(c_{2,p},n)$, where $n \in \mathbb{N}$.
Related fusion rules for $p=3$ were previously analyzed in \cite{M1}.

Here is the main result

\begin{theorem} \label{fr-p}
The space of intertwining operators
\bea \label{03p2} && I { L^{Vir}(c_{2,p},h) \choose
L^{Vir}(c_{2,p},h_{5,1}) \ \ L^{Vir}(c_{2,p},h_{2n+3,1} ) }\eea
is nontrivial only if
\be \label{03p2-fr} h \in \left\{ h_{2n-1,1}, h_{2n+1,1},
h_{2n+3,1}, h_{2n+5,1}, h_{2n+7,1}
%
%
 \right\}. \ee

\end{theorem}

{\em Proof.} Proofs of similar results have already appear in the
literature so here we do not provide full details (see \cite{M2} for
instance).

We start from intertwining operator $\mathcal{Y}$ of type as in
(\ref{03p2}), and consider highest weight vectors $u \in L^{Vir}(c_{2,p},h)$, $v \in L^{Vir}(c_{2,p},h_{5,1})$
and $w \in L(c_{2,p},h_{2n+3,1})$.Then the matrix coefficient
$$\langle w', \mathcal{Y}(u,x) v \rangle$$ is
proportional to $x^{h-h_{2n+3,1}-h_{5,1}}$. It is known (see \cite{FF}) that
$L^{Vir}(c_{2,p},3p-2)$ can be obtained as a quotient of the Verma module
$M^{Vir}(c_{2,p},3p-2)$ by the submodule generated by a pair of singular vectors, where one
singular vector is of (relative) degree $5$. While one can analyze both vectors
explicitly, for our purposes we only analyze the singular vector of
degree five. It is not hard to see that
$$v_{sing}=\biggl(L(-1)^5-10pL(-2)L(-1)^3+(21p^2-15p)L(-3)L(-1)^2$$
$$+16p^2 L(-2)^2L(-1)+(42p^2-18p-36p^3)L(-4)L(-1)+(-24p^3+16p^2)L(-3)L(-2)$$
$$+(-66p^3+46p^2+36p^4-12p)L(-5) \biggr){\bf 1} \in M^{Vir}(c_{2,p},3p-2)$$
is such a vector (unique up to a constant). This singular vector
give rise to a differential equation of degree $5$ satisfied by
$x^{h-h_{2n+3,1}-h_{5,1}}$. This follows from the relations
$$[L(-n),\mathcal{Y}(u,x)]=(x^{-n+1} \frac{d}{dx}-(n-1) h_{2n+3,1} x^{-n})\mathcal{Y}(u,x), \ \ n \geq 1.$$
The differential equation obtained (we omit an explicit formula
here) can be now solved and the $h$-solutions are given in
(\ref{03p2-fr}). \epfv

\begin{remark} {\em To prove the "if" part in the theorem we would have to analyze both singular
vectors and describe $A(V^{Vir} (c_{2,p},0))$-bimodule
$A(L^{Vir}(c_{2,p},3p-2))$ (cf. \cite{M2}).

}
\end{remark}

\section{The vertex operator algebra $\WW_{2,p}$}

We define the following elements in $V_L$:
$$ a ^- = Q e^{-2\a}, \quad a ^{+} = G Q e ^{-2 \a}.$$
Since $Q$ anti-commutes with itself, we have $Q^2=0$ and hence $Q
a^-=0$ and $Q a^+=0$. It is not hard to see that $G Q e^{-2 \a} \neq 0$.
Since $\widetilde{Q}$ annihilates $e
^{-2\a}$ and commutes with $Q$ and $G$, we conclude  that $a^\pm
\in \overline{V_L}$. A related vertex operator superalgebra has
been encountered in our study of the triplet vertex algebra and is
a subject of our forthcoming paper \cite{AdM-2009-2} (see also
\cite{FFT}).

Let $V$ be the subalgebra of $\overline{V_L}$ generated by $1, a
^{-}, a ^{+}$.

The following lemma is useful for constructing singular vectors in
$\overline{V_L}$:
\begin{lemma} \label{non-triv} We have:
\item[(i)] $Y(a ^-, x) Q e ^{ - (n+1) \a } \in W ((x)), $
where $$W = U  (Vir). Q e^{-(n+2) \a} \cong L ^{Vir}(c_{2,p},
h_{2n+5,1}).$$

\item[(ii)] $ G^{n} Q e^{-(n+1) \a} \ne 0 \quad  \mbox{for} \ n
\in {\N}$.
\item[(iii)] $ G^{n} Q e^{-(n+1) \a} \in V$.
\end{lemma}
{\em Proof.}
First we shall prove assertion (i).
Define $j_0 = - n
p + p-2$. Then we have
$$ a^- _ j Q e ^{- (n+1) \a} = 0 \quad \mbox{for} \ j > j_0, $$
$$ a ^- _{j_0} Q e^{ - (n+1) \a } = \mu_n  Q e ^{- (n+2) \a } \quad
( \mu_n \ne 0). $$
Let us now see that
\bea \label{rel-nt1} &&w_{n}:=G Q e ^{-(n+1) \a} \ne 0. \eea
In what follows we shall prove that $w_2=H= G Q e ^{-3 \alpha} \ne 0$ (cf. Theorem \ref{CT}).
Then the relation
$$ G( a ^- _{-2} a ^- ) = a ^+ _{-2} a ^- + a ^- _{-2} a ^+ = \mu _1 H$$
gives that $w_1 = a ^+ \ne 0$.
Our Theorem \ref{G-der-general} implies
\bea \label{rel-nt2} e ^{ (n-1) \alpha }_{j_1} w_n = (-1) ^{n-1} a ^+, \quad (j_1 = (n+1) (n-1) p -1),  \eea
so we have that $w_n \ne 0$ for every $ n \in {\N}$.

Assume that $ j \le j_0$. Then by using Proposition \ref{fr-p} we
see that
$a ^ - _j Q e ^{-(n+1) \a}$ lies in the Virasoro-submodule of
$M(1) \otimes e ^{-(n+1) \a}$ generated by (co)singular vectors of
weights $h_{2n+5,1}$ and $h_{2n+7,1}$. But cosingular vector of
weight $h_{2n+7,1}$ is proportional to $G e ^{-(n+3) \a} \notin
\overline{V_L}$ (since by (\ref{rel-nt1}) we have $w_{n+2}= Q G e ^{-(n+3) \a } \ne 0$).
Therefore $$a ^ - _j Q e ^{-(n+1) \a} \in U(Vir) Q e ^{-(n+2)
\a}.$$ In this way we have proved assertion (i).

 We shall prove the assertions (ii) and (iii) by induction on $n \in
{\Zp}$.

For $n=1$ we have already proved that $a ^{+}$ is a nontrivial
singular vector.

Assume now that assertions (ii)-(iii)  hold for certain $n \in
{\Zp}$. Since $V_{L} $ is a simple vertex operator algebra we have
that
$$ Y(a ^{+},z) G^{n} Q e^{-(n+1) \a}\ne 0,$$
(for the proof see \cite{DL}).
So there is $k_0 \in {\Z}$ such that
$$a ^+ _{k_0} G^{n} Q e^{-(n+1) \a}   \ne 0 \quad \mbox{and} \quad
 a ^+ _{j} G^{n} Q e^{-(n+1) \a}  = 0 \ \mbox{for} \ j > k_0.$$
Since
$$ a ^+ _{k_0} G^{n} Q e^{-(n+1) \a} = \nu_2 G^{n+1} (a ^{-} _{k_0}
 Q e^{-(n+1) \a} ), $$
for certain non-zero constant ${\nu}_2$, then by using assertion
(i) and the fact that $G$ is a screening operator we conclude that
$$ a ^+ _{k_0} G^{n} Q e^{-(n+1) \a}  \in U(Vir) G ^{n+1} Q e
^{-(n+2) \a}.$$
Therefore
$G ^{n+1} Q e ^{- (n+2) \a } \ne 0$ and
$$ G ^{n+1} Q   e ^{-(n+2) \a }  = \frac{1}{{\nu}_2 {\mu_n}} a ^+ _{j_0} G^{n} Q e^{-(n+1) \a} \in V.$$
The proof follows.
 \qed

By using Lemma \ref{non-triv} and the structure theory of
Feigin-Fuchs modules from \cite{Fel},  \cite{FF} and \cite{TK} one
can see that the following theorem holds for every $p \geq
3$:

 \begin{theorem} \label{str-ff-1}
As a Virasoro algebra module, $V_{L}$  is generated by the family
of (non-trivial) singular $\widetilde{Sing}$ vectors, two families of subsingular  vectors
$\widetilde{SSing}^{(1)} \cup \widetilde{SSing} ^{(2)} $, and a family of cosingular vectors $\widetilde{CSing} ^{(3)}$
 where
\bea &&  \widetilde{Sing} =  \{  u _{(j, n)},   \ \vert \ j, n \in
{\N}, \ 0 \le j \le n \} \nonumber \\
  && \widetilde{SSing}^{(1)} =  \{ w ^{(1)} _ {(j, n)},    \ \vert \ n \in {\Zp}, 0 \le j \le n \}
  \nonumber \\
&& \widetilde{SSing}^{(2)} =  \{ w ^{(2)} _ {(j, n)},    \ \vert \
n \in {\Zp}, 0 \le j \le n-1 \}
  \nonumber \\
&& \widetilde{CSing}^{(3)} =  \{ w ^{(3)} _ {(j, n)},    \ \vert \
n \in {\Zp}, 0 \le j \le n-1 \}.
  \nonumber
\eea
These vectors satisfy the following conditions:
\bea
&&  u _{(j, n)}= G ^{j} Q e ^{-(n+1) \a} \in M(1) \otimes e
^{(2 j-n) \a}, \nonumber \\
&&  w^{(1)} _{(j, n)}= G ^{j} e ^{-(n+1) \a} \in M(1) \otimes e
^{(2 j-n-1) \a}, \nonumber \\
&& w^{(2)} _ {j,n} \in M(1) \otimes e ^{(n - 2j-1) \a}, \ \quad
\quad    G ^j Q
w_{j,n}^{(2)} = e ^{n \a}, \nonumber \\
&& w^{(3)} _ {j,n} \in M(1) \otimes e ^{(n - 2j) \a}, \quad G ^j
w_{j,n}^{(3)} = e ^{n \a}. \nonumber
 \eea
\end{theorem}

\begin{remark}
The main novelty of our approach is in the fact that we use
the screening operator $G$ for which we have explicit formulae in the
context of vertex operator algebras. Screening operators from
\cite{Fel} and \cite{TK} are defined by using (complex) contour integrals.
In the above case these two approaches are equivalent.
\end{remark}

We should mention here that $G ^{j} e ^{-(n+1) \a} \in
\widetilde{SSing}^{(1)}$ become singular in $V_L / {\rm Ker} \ Q$.
But these vectors are not annihilated by $Q$. So we get


 \begin{proposition} We have
\item[(i)]As a Virasoro module, $\overline{V_L}$ is generated by the vacuum
vector ${\bf 1}$ and the family of singular vectors
$$ \{ G ^j Q e ^{-(n+1) \a} \ \vert \ j \in \N,  n \in {\Zp},  \ j \le n \}.$$
Moreover,
$$ \overline{V_L} = V^{Vir} (c_{2,p},0) \bigoplus \bigoplus_{n=1 }^{\infty}
(n+1) L^{Vir}(c_{2,p},\frac{(n+1)(pn+2p-2)}{2}). $$

\item[(ii)] We have $V=\overline{V_L}$, i.e., the vertex algebra $V=\overline{V_L}$ is strongly  generated by
${\bf 1}$, $\omega$ and
$$ a^- = Q e^{-2 \a} \quad \mbox{and} \quad a^{+} = G a ^{-} . $$
 \end{proposition}
\begin{proof}
First we notice that
$$ Q e ^{-n \a}\ne 0, \widetilde{Q} e ^{- n \a} =0, \quad \widetilde{Q} e ^{n \a} \ne  0,  \quad ( n \in \Zp).$$
Since the operators $Q,  \widetilde{Q}$ and $G$ mutually commute we
have that only singular vectors $u _{(j, n)}$ are annihilated by
both $Q$ and $\widetilde{Q}$. This proves assertion (i). The
assertion (ii) easily follows from Lemma \ref{non-triv}.
\end{proof}

 Define the following (nonzero)  vectors in the vertex algebra $V_D$:
 $$ F= Q e ^{-3 \a}, \quad H  = G F, \quad E = G ^2 F. $$
 Again, it is clear that $F$ (and $H$ and $E$) are inside $\WW_{2,p}$.

In the same way as above (see also \cite{AdM-triplet}) we get the
following result.

\begin{proposition}
\item[(i)]As a Virasoro module, $\WW_{2,p}$ is generated by the  vacuum
vector ${\bf 1}$ and the family of singular vectors
$$ \{ G ^j Q e ^{-(2n+1) \a} \ \vert \ j \in \N,  n \in {\Zp},  \ j \le 2 n \}.$$
Moreover,
$$ \WW_{2,p} = V^{Vir} (c_{2,p},0) \bigoplus \bigoplus_{n=1 }^{\infty}
(2 n+1) L^{Vir}(c_{2,p},(2n+1)(pn+p-1)). $$

\item[(ii)] As a vertex algebra, $\WW_{2,p}$ is strongly  generated by
${\bf 1}, \omega,  E, F, H $
 \end{proposition}

 The structure of $\WW_{2,p}$  and $\overline{V_L}$ can be visualized via the embedding diagram for the Virasoro module $V_L$ (cf. \cite{Fel}, \cite{FF}):
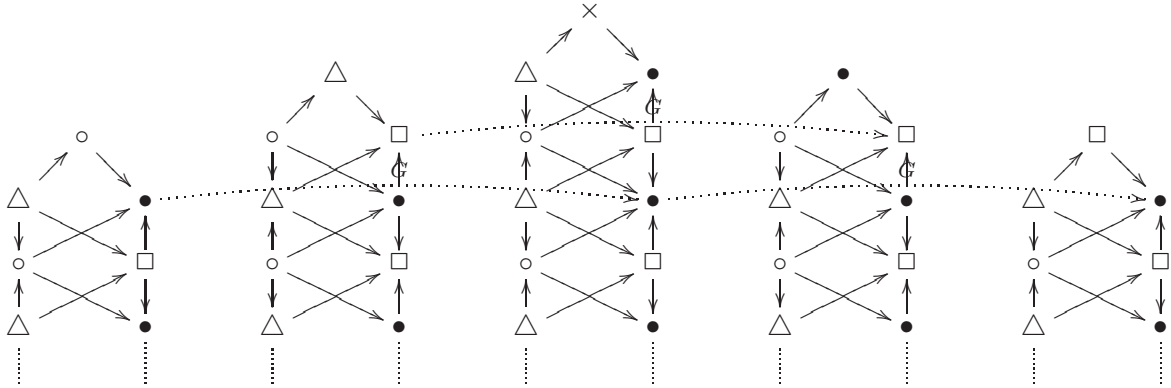
 $$
 \xymatrix@!0{& & & &     & &  &  &       &  & \times  \ar[dr] &  &       & & & &    & & & & \\
 & & & &      &  & \bigtriangleup  \ar[dr] &  &       &   \bigtriangleup  \ar[ur]  \ar[rrd]  \ar[d] & &  \bullet &          &  & \bullet  \ar[dr] &  &      & & & &  \\
  & & \circ  \ar[dr] &  &    &   \circ  \ar[ur]  \ar[rrd]  \ar[d] & &  \Box   \ar@{.>}@/^/[rrrrrrrr]^G &         & \circ \ar[rru] \ar[rrd] &  & \Box \ar[u] \ar[d] &
  &   \circ  \ar[ur]  \ar[rrd]  \ar[d] & &  \Box  &        & & \Box  \ar[dr] &  &    \\
   &   \bigtriangleup  \ar[ur]  \ar[rrd]  \ar[d] & &  \bullet   \ar@{.>}@/^/[rrrrrrrr]^G  &        & \bigtriangleup \ar[rru] \ar[rrd] &  & \bullet \ar[u] \ar[d] &
   &   \bigtriangleup  \ar[u]  \ar[rru]  \ar[rrd]  \ar[d] & &  \bullet  \ar@{.>}@/^/[rrrrrrrr]^G &        & \bigtriangleup \ar[rru] \ar[rrd] &  & \bullet \ar[u] \ar[d] &
      &   \bigtriangleup  \ar[ur]  \ar[rrd]  \ar[d] & &  \bullet & \\
 & \circ \ar[rru] \ar[rrd] &  & \Box \ar[u] \ar[d] &       &   \circ  \ar[u]  \ar[urr]  \ar[rrd]  \ar[d] & &  \Box &
  &  \circ \ar[rru] \ar[drr] &  & \ar[u] \Box \ar[d] &     &   \circ  \ar[u]  \ar[urr] \ar[rrd]  \ar[d] & &  \Box &
   & \circ \ar[rru] \ar[rrd] &  & \Box \ar[u] \ar[d] & \\
  & \bigtriangleup \ar[urr]  \ar[u]  \ar@{.}[d] &  & \bullet  \ar@{.}[d] &
   &  \bigtriangleup \ar[rru] \ar@{.}[d] &  & \ar[u] \bullet \ar@{.}[d] &
   & \bigtriangleup \ar[urr]  \ar[u]  \ar@{.}[d] &  & \bullet \ar@{.}[d]  &
   &  \bigtriangleup \ar[rru] \ar@{.}[d] &  & \ar[u] \bullet \ar@{.}[d] &
  & \bigtriangleup \ar[urr]  \ar[u]  \ar@{.}[d] &  & \bullet  \ar@{.}[d] &  \\
 &&&&  & & & &  & & & &   & & & &   & & & &
 }
 $$
The left-most "ladder" is $M(1) \otimes e^{-2 \alpha} \in V_D$, followed by $M(1) \otimes e^{-\alpha} \in V_{D + \alpha}$, the middle module is $M(1) \in V_D$, $M(1)
\otimes e^{\alpha} \in V_{D+\alpha}$ and $M(1) \otimes e^{2 \alpha} \in V_D$.  Singular vectors in $V_D$  are denoted by $\bullet$, whereas in $V_{D+\alpha}$ singular
vectors  are denoted by $\Box$.  These vectors account for $\widetilde{Sing}$ set introduced earlier.  Similarly,  vectors denoted by $\circ$  in $V_{D}$ and
$\bigtriangleup \in V_{D+\alpha}$ give the set $\widetilde{SSing}^{(1)}$. Dotted arrows indicate the action of the $G$-screening.

\section{Singlet vertex algebra $\overline{M(1)}$ }

In this section we shall study representation theory of the
singlet vertex algebra, which is a subalgebra of $\WW_{2,p}$. The
results from this section will be used in the proof of
$C_2$--cofiniteness of $\WW_{2,p}$.

 As in \cite{A-2003} and \cite{AdM-2007} we have the singlet vertex
operator algebra
$$ \overline{M(1)} = \mbox{Ker}_{M(1)}  Q \cap \mbox{Ker}_{M(1)}
\widetilde{Q},$$
which is a subalgebra of
$\WW_{2,p}$. The structure of $M(1)$ as a module for the
Virasoro algebra can be described by using Theorem \ref{str-ff-1}.
We have the following result which is analogous to the results
obtained in \cite{A-2003} in the case of $(1,p)$ Virasoro modules.

\begin{proposition}
\item[(i)]As a Virasoro module, $ \overline{M(1)}$ is generated by the  vacuum
vector ${\bf 1}$ and the family of singular vectors
$$ \{ G ^n Q e ^{-(2n+1) \a} \ \vert \  n \in {\Zp} \}.$$
Moreover,
$$ \overline{M(1)} = V^{Vir} (c_{2,p},0) \bigoplus \bigoplus_{n=1 }^{\infty}
 L^{Vir}(c_{2,p},(2n+1)(pn+p-1)). $$

\item[(ii)] As a vertex algebra, $ \overline{M(1)}$ is strongly generated by
${\bf 1}$, $\omega$  and $H $.
 \end{proposition}

As in \cite{A-2003} and \cite{AdM-triplet} we have the following useful proposition:

\begin{proposition} \label{korisna-pr} We have
\item[(i)] $V ^{Vir} (c_{p,2},0) \cong \mbox{Ker}_{M(1)} G. $
\item[(ii)] $H_i H \in U( Vir). {\bf 1} \quad \mbox{ for every} \  i \ge -3p-1.$
\item[(iii)] $H_{-3p -2} H = C G ^2 Q e ^{-5 \a} + U(Vir) .{\bf 1} \quad (C \ne 0). $
\end{proposition}

Now we want to classify all irreducible
$\overline{M(1)}$--modules. We shall use similar approach as in
\cite{A-2003}. So we need to evaluate certain singular vectors for
the Virasoro algebra on top components of $M(1)$--modules
$M(1,\l)$. It turns out that this case involve more complicated
combinatorics.

 We shall start with the general following theorem which gives
non-triviality of action of singular vectors for the Virasoro
algebra on top components of $M(1)$--modules.

\begin{theorem}
Assume that $u \in M(1)$ is a singular vector for the Virasoro
algebra. Let $u(0):=u_{\deg u -1}$. Assume that  $M(1, \l)$ is a
$M(1)$--module, which is irreducible as a module for the Virasoro
algebra with central charge $c_{2,p}$ (In other words $M(1,\l)$ is
an irreducible Feigin-Fuchs module). Then
$$u(0) v_{\l} = {\nu} v_{\l} \quad \mbox{for certain} \ \ \nu \ne 0.$$
\end{theorem}
\begin{proof}
Since $M(1,\l)$ is an irreducible module for the simple vertex
algebra $M(1)$ we have that
$$ Y(u,z) v_{\l} \ne 0,$$
(for the proof see \cite{DL}). So there is $j_0 \in {\Z}$ such
that
$$u _{j_0} v_{\l}   \ne 0 \quad \mbox{and} \quad
 u_{j} v_{\l}  = 0 \ \mbox{for} \ j > j_0.$$
 Since $u$ and $v_{\l}$ are singular vectors for the Virasoro
 algebra we have that
$$L (n) u _{j_0} v_{\l}= 0 \quad \mbox{for} \ n \ge 1.$$
Therefore $u_{j_0} v_{\l}$ is a non-trivial singular vector in the
irreducible Feigin-Fuchs module $M(1,\l)$ and  conclude that it is
proportional to $v_{\l}$. Therefore $j_0 = \deg u -1$ and
$u(0) v_{\l} = {\nu} v_{\l}$ for certain non-zero constant $\nu$.
\end{proof}

By using expression (\ref{G-log}) for the operator $G$ and similar calculation to that  in \cite{A-2003}, \cite{AdM-2007}, \cite{AdM-triplet}, \cite{AdM-sigma} we get:

\begin{proposition} \label{evaluation}
\item[(1)] Assume that $v_{\l}$ is a lowest weight vector in
$M(1)$--module $M(1, \l)$ and $ t = \l (\a)$. Then we have:

\bea &&H(0)v_{\l}= -( \mbox{Res}_{x_1,x_2,x_3} \nonumber \\ && \
{\rm ln}(1-x_2/x_1) (x_1 x_2 x_3) ^{-3p} (x_1-x_2) ^p (x_1 -x_3)
^p (x_2-x_3) ^p (1+x_1) ^t (1+x_2) ^t (1+x_3) ^t )\
v_{\l}.\nonumber \eea
%
%
\item[(2)]
Let $v = G ^2 Q e ^{-5\a}$. Then
$$o (v) v_{\l} =\widetilde{G_p}(t) v_{\l} \ne  0$$ where
\bea && \widetilde{G_p}(t) =    \mbox{Res}_{x_1,x_2,x_3,x_4,x_5} ({\rm ln}(1-x_2/x_1) {\rm ln}(1-x_4/x_3)(x_1 x_2 x_3 x_4 x_5) ^{-5p}\nonumber \\
&&  \Delta(x_1,x_2,x_3,x_4,x_5)^p (1+x_1) ^t (1+x_2) ^t (1+x_3) ^t
(1+x_4) ^t (1+x_5) ^t  ) . \nonumber \eea
(Here $\Delta(x_1,x_2, \dots,x_n) ^p$ denotes the $p$--th power of the Vandermonde determinant.)
\end{proposition}

The next result appears to be fairly difficult to prove.
\begin{theorem} \label{CT} Assume that $v_{\l}$ is a lowest weight vector in $M(1)$--module $M(1, \l)$ and $ t = \l (\a)$. Then we have:
$$ H(0) v_{\l} = A_p { t + p \choose 2p-1} {t \choose 2p -1} { t + p/2
\choose 2p -1} v_{\l} \quad (A_p \ne 0).$$
\end{theorem}
This theorem will be proven in the appendix, where we also
compute the constant $A_p$ (this is not really needed because $A_p \neq 0$ is everything we need).

Now we are in the position to determine the Zhu's algebra of
$\overline{M(1)}$ and therefore classify their  irreducible
$\N$--graded modules. By using Proposition \ref{korisna-pr}, Theorem \ref{CT} and similar arguments to that  of Theorem 6.1 of \cite{A-2003} we get:
\begin{theorem} \label{zhu-singlet-2-p}
Zhu's algebra of the singlet vertex algebra $A(\overline{M(1)})$
is isomorphic to the commutative associative algebra
$$A(\overline{M(1)}) \cong \frac{{\C}[x,y]}{ \langle P(x,y) \rangle }$$
where
$$ P(x,y) = y ^2 - C_p  \prod_{i=0} ^{{2p-2}} \left( x- \frac{(2p-2-i) (p-i)}{2p}\right) ^2 \prod_{i = 0} ^{2p-2} \left( x- \frac{(3p-4-2i) (p-2i)}{8p} \right)\qquad
(C_p \ne 0).$$

\end{theorem}

In particular,  for $p=3$ the polynomial $P(x,y)$ is given by

$$y^2-C_3(x+\frac{1}{24})(x-\frac{5}{8})^2(x-\frac{1}{8})^2 x^4 (x-1)^2 (x-2)^2(x-\frac{1}{3})^2$$

where $C_3$ is a constant.

\vskip 5mm

\section{The $C_2$-cofiniteness of  vertex algebras  $\WW_{2,p}, p \geq 3$}

As usual, for a vertex operator algebra $V$ we let
$$C_2(V)=\{ a_{-2}b : a, b \in V \}.$$
It is a fairly standard fact (cf. \cite{Zhu}) that $V/C_2(V)$ has a
Poisson algebra structure with the multiplication
$$\bar{a} \cdot \bar{b}=\overline{a_{-1} b},$$
where $-$ denotes the natural projection from $V$ to $V/C_2(V)$. If
${\rm dim}(V/C_2(V))$ is finite-dimensional we say that $V$ is
$C_2$-cofinite.

\begin{lemma}  \label{relacije}In $\WW_{2,p}$ we have the following relations
\bea && E_{-1} F + F_{-1} E + 2 H_{-1} H =0, \label{EFH} \\
&& E_{-1} E = F_{-1} F = 0 . \label{rel-nilp}
\eea
\end{lemma}
\begin{proof}
Relation $F_{-1} F = 0$ follows from definition. Since $E_{-1} E $
is proportional to $G ^4 (F_{-1}F)$ we get (\ref{rel-nilp}). Then we
have that
$$E_{-1} F + F_{-1} E + 2 H_{-1} H = G ^2 ( F_{-1} F) = 0.$$
The proof follows.
\end{proof}

As usual, for a vertex operator algebra $V$ we let
$$C_2(V)=\{ a_{-2}b : a, b \in V \}.$$
It is a fairly standard fact (cf. \cite{Zhu}) that $V/C_2(V)$ has a
Poisson algebra structure with the multiplication
$$\bar{a} \cdot \bar{b}=\overline{a_{-1} b},$$
where $-$ denotes the natural projection from $V$ to $V/C_2(V)$. If
${\rm dim}(V/C_2(V))$ is finite-dimensional we say that $V$ is
$C_2$-cofinite.

The aim of this section is the following result
\begin{theorem} \label{c2-general}
The vertex operator algebra $\WW_{2,p}$ is $C_2$-cofinite.
\end{theorem}
\noindent {\em Proof.} First we notice $\WW_{2,p}/C_2(\WW_{2,p})$ is
generated by the set $\{ \bar{E}, \bar{F}, \bar{H}, \bar{\omega}
\}$. By using Lemma \ref{relacije} and commutativity of $\WW_{2,p}
/C_2(\WW_{2,p})$, we obtain that
$$\bar{E} ^{2} = \bar{F} ^{2} = 0,$$
and

$$ \bar{H} ^{2} = - \bar{E} \bar{F}$$
which implies that
$$\bar{H} ^{4}=0.$$

Moreover, the description of Zhu's algebra from Theorem
\ref{zhu-singlet-2-p} implies that
$$ \bar{H} ^{2} = C\bar{\omega} ^{6 p -3}, \quad (C \ne 0).$$
Since $\bar{H} ^{4} = 0$, we conclude that $\bar{\omega} ^{12p
-6}=0$. Therefore, every generator of the commutative algebra
$\WW_{2,p} /C_2(\WW_{2,p})$ is nilpotent and therefore $\WW_{2,p}
/C_2(\WW_{2,p})$ is finite-dimensional.
 \qed

\section{Classification of irreducible $\WW_{2,3}$-modules}
\label{klasifikacija-p3}

In this section we shall consider the case $p=3$. We shall construct
and classify all irreducible modules for $\WW_{2,3}$. It turns out
that the proof of classification result is analogous to the case of
triplet vertex algebra $\WW(p)$ from \cite{AdM-triplet}.

\begin{theorem} \label{zhu1}
In Zhu's algebra $A(\WW_{2,3})$ we have the following relation
$$f([\omega])=0,$$
where \bea \label{relzhu1} f(x)=&& x ^3  \left( (x-1) (x-2) (x
-\tfrac{1}{8}) (x-\tfrac{5}{8}) (x-\tfrac{1}{3}) \right) ^2
\nonumber \\ && (x-5) (x-7) (x-\tfrac{10}{3})(x+\tfrac{1}{24})
(x-\tfrac{33}{8})
  (x-\tfrac{21}{8}) (x-\tfrac{35}{24}) . \eea
\end{theorem}
\begin{proof}
 First we notice that
 $ Q e ^{-5 \a} \in O(\WW_{2,p})$, which implies that
 $$ v=G ^2 Q e ^{-5 \a} \in O(\WW_{2,p}) \quad \mbox{and} \quad [v] = 0 \in A(\WW_{2,p}).$$
By using structure of $\WW_{2,p}$ and similar arguments as  in
\cite{AdM-triplet} we conclude that
$$[v] = f ([\omega]) $$ for certain polynomials of
degree $20$. This polynomials can be determined by evaluating $o(v)
v_{\l}$. Direct calculation (note that in the case p=3 Conjecture \ref{conj-g} holds) and Proposition \ref{evaluation} give
explicit formula (\ref{relzhu1}).
\end{proof}
Define the ideal:
$$\mathcal{R}_{2,3} = \WW_{2,3}. \omega.$$

\begin{proposition} \label{ideal} We have:
\item[(i)] $\mathcal{R}_{2,3}$ is a non-trivial ideal in $\WW_{2,3}$
and $\WW_{2,3} / \mathcal{R}_{2,3} = {\C} {\bf 1}$.
\item[(ii)]$\mathcal{R}_{2,3}$ is an irreducible
$\WW_{2,3}$--module.
\end{proposition}
\begin{proof}
First we notice that
$$ L(n+1) \omega=  X (n) \omega = 0 \quad \mbox{for} \ X \in \{E, F, H \}, \ n \in
{\N}.$$ (Here $X(n) : = X_{14+n}$.)
This easily implies that ${\bf 1} \notin \mathcal{R}_{2,3}$. So,
$\mathcal{R}_{2,3} \ne \WW_{2,3}$. Clearly, $E, F, H \in
\mathcal{R}_{2,3}$, which implies that
$$  G ^j Q e ^{-(2n+1) \a} \in \mathcal{R}_{2,3} \quad \mbox{for} \
n \in {\Zp}, \ j \le 2n. $$
This proves assertion (i). Assume now that $N \subset
\mathcal{R}_{2,3}$ is a non-trivial submodule. Then $N$ is a
${\N}$--graded, and top component $N(0)$ must have conformal weight
which is greater than $2$. Now Theorem \ref{zhu1} implies that top
component $N(0)$ of $N$ must have conformal weight $5 $ or $7$. This
is a contradiction, since every vector in $N(0)$ is a singular
vector for the Virasoro algebra and since there are no singular
vectors of conformal weights $5$ and $7$ in $\mathcal{R}_{2,3}$. The
proof follows.
\end{proof}

Define the set

$$ S_{2,3} = S_{2,3} ^{(1)} \cup S_{2,3} ^{(2)},$$
where
$$S_{2,3} ^{(1)} = \{0, 1, 2,  \tfrac{1}{3}, \tfrac{1}{8},
\tfrac{5}{8}, -\tfrac{1}{24} \}  \quad S_{2,3} ^{(2)} = \{ 5, 7,
\tfrac{10}{3}, \tfrac{21}{8}, \tfrac{33}{8}, \tfrac{35}{24} \} .$$

Define the following $\WW_{2,3}$--modules:

\bea
&&W(2)=\mathfrak{X}^+_{1,1}   = \mathcal{R}_{2,3}=\WW_{2,3}. \omega \subset \WW_{2,3},  \nonumber \\
&&W(0) =
\WW_{2,3} / \mathfrak{X}^+_{1,1} , \nonumber \\
&& W(7)=\X^-_{1,1} =\WW_{2,3} . (e ^{- 2 \a}+ \mbox{Ker}_{V_L} Q)\subset V_L / \mbox{Ker}_{V_L} Q, \nonumber \\
&& W(1)=\X^+_{2,1}= \WW_{2,3}.  e ^{\a}   \subset V_L \nonumber \\
&& W(5)=\X^-_{2,1} = \WW_{2,3}. ( e ^{2\a} + \overline{V_L} ) \subset V_L /
\overline{V_L} \nonumber \\
&& W(- \tfrac{1}{24}) =\X^+_{3,2}  = \WW_{2,3} . e ^{\a /6} \subset V_{L + \a /6}
\nonumber \\
&&W(\tfrac{1}{8})=\X^+_{2,2} = \WW_{2,3} . e ^{\a /2} \subset V_{L + \a /2}
\nonumber \\
&&W (\tfrac{1}{3})={\X}^+_{3,1}    = \WW_{2,3} . e ^{2 \a /3} \subset V_{L + 2\a /3}
\nonumber \\
&&W (\tfrac{5}{8})=\X^+_{1,2}  = \WW_{2,3} . e ^{ 5 \a /6} \subset V_{L + 5 \a
/6}
\nonumber \\
&& W(\tfrac{35}{24})=\X^-_{3,2} = \WW_{2,3} . e ^{ 7 \a /6} \subset V_{L + 7 \a
/6}
\nonumber \\
&&W(\tfrac{21}{8})= \X^-_{2,2} = \WW_{2,3} . e ^{ 3 \a /2} \subset V_{L + 3 \a
/2}
\nonumber \\
&&W(\tfrac{10}{3})= \X^-_{3,1} = \WW_{2,3} . e ^{ 5 \a /3} \subset V_{L + 5 \a
/3}
\nonumber \\
&&W(\tfrac{33}{8})= \X^-_{1,2} = \WW_{2,3} . e ^{ 11 \a /6} \subset V_{L + 11
\a /6}. \nonumber
\eea

We also indicate the ${\X}$-parametrization for the modules following
\cite{FGST-log}. In the parenthesis next to a module we denote the
lowest conformal weight.

\begin{theorem} \label{konstr-ired-2-p}
For every $ h \in S_{2,3} $ $W(h)$ is a $\mathbb{N}$-graded
irreducible $\WW_{2,3}$--module whose top component $W(h) (0)$ has
lowest conformal weight $h$. Moreover, $W(h) (0)  $ is
$1$-dimensional if $h \in S_{2,3} ^{(1)}$ and $2$-dimensional if $h
\in S_{2,3} ^{(2)}$.
\end{theorem}
\begin{proof}
The proof is similar to that of the irreducibility result obtain in
Proposition \ref{ideal} for $\mathcal{R}_{2,3}$ and to that of
Theorem 3.12 of \cite{AdM-triplet}. So we omit some technical
details. First we see that each $W(h)$ is a $\mathbb{Z}_{\geq
0}$--graded $\WW_{2,3}$--module whose top component is irreducible
module for the Zhu's algebra $A(\WW_{2,3})$. Then by using Theorem
\ref{zhu1} and simple analysis of weights of modules $W(h)$, we see
that they are irreducible $\WW_{2,3}$--modules.
\end{proof}

As in \cite{AdM-triplet} we have the following result on the
structure of Zhu's algebra $A(\WW_{2,3})$.

\begin{proposition} \label{zhu-komut-2-p}
The Zhu's associative algebra $A(\WW_{2,3})$ is generated by $[1], [\omega], [E], [F] $ and $[H]$.
  We have the following relations:
  \bea &&
[H]*[F]-[F]*[H]=-2q([\omega])[F], \\
&& [H]*[E]-[E]*[H]=2q([\omega])[E] \\
&& [E]*[F]-[F]*[E]=-2q([\omega])[H]. \eea where $q$ is a certain
polynomial. \end{proposition}

\begin{theorem}
The set
$$ \{ W(h), \ h \in S_{2,3} \}$$
provides, up to isomorphism, all irreducible modules for the vertex
operator algebra $\WW_{2,3}$.
\end{theorem}
\begin{proof} The proof is also similar to that of the classification result in
\cite{AdM-triplet}. It is enough to see that $W(h) (0)$, $h \in
S_{2,3}$ gives all irreducible modules for Zhu's algebra
$A(\WW_{2,3})$.
 Assume that $U$ is an irreducible $A(\WW_{2,3})$--module. Relation
$f ([\omega]) = 0$ in $A(\WW_{2,3})$ implies that
$$L(0) \vert U = h \ \mbox{Id}, \quad \mbox{for} \quad h  \in S_{2,3}.$$

Assume first that $ h \in S_{2,3} ^{(2)}$. By combining Proposition
\ref{zhu-komut-2-p} and Theorem \ref{konstr-ired-2-p} we have that
$q(h) \ne 0$. Define
$$ e= \frac{ 1}{ \sqrt{2} q(h)} E, \quad f= -\frac{1}{ \sqrt{2}q(h)} F, \quad h= \frac{1}{q(h)} H .$$

Therefore $U$ carries the structure of an irreducible,
$\mathfrak{sl}_2$--module with the property that $e^2 = f^2 = 0$  and $h
\ne 0$ on $U$. This easily implies that $U$ is $2$--dimensional
irreducible $\mathfrak{sl}_2$--module. Moreover,  as an
$A(\WW_{2,3})$--module $U$ is isomorphic to  $W(h) (0)$.

Assume next that $ h \in S_{2,3} ^{(1)}$. If $q(h) \ne 0$, as above
 we conclude that $U$ is an irreducible $1$--dimensional $sl_2$--module. Therefore $U \cong W(h) (0)$.

 If $q(h)=0$
 from Proposition \ref{zhu-komut-2-p} we have that the action of generators of $A(\WW_{2,3})$ commute on $U$. Irreducibility of
  $U$ implies that $U$ is $1$-dimensional. Since $[H], [E] ^2, [F] ^2$ must act trivially on $U$, we conclude that $[H], [E] , [F] $ also act trivially on $U$. Therefore
  $U \cong W(h) (0)$.
\end{proof}

\section{Description of the space of generalized characters for $\WW_{2,3}$}

By using Proposition \ref{korisna-pr}, Theorem \ref{zhu1} and the proof similar to that of   Corollary
3.6 in \cite{AdM-triplet} we obtain the following useful consequence.

\begin{proposition} \label{poisson} Inside the Poisson algebra $V/C_2(V)$, we have the relation $\overline{\omega}^{20}=0$.
\end{proposition}

Next result is essentially from \cite{FGST-log}

\begin{lemma} \label{slclosure} The $SL(2,\mathbb{Z})$-closure of the space of irreducible characters is $20$-dimensional.
\end{lemma}

Proposition \ref{poisson}, Lemma \ref{slclosure} and the results from \cite{AdM-2009-3} now give

\begin{theorem} Any generalized $\WW_{2,3}$ character satisfies a unique modular differential equation of order $20$.
Consequently, the space of generalized characters is precisely $20$-dimensional and coincides with the space in Lemma \ref{slclosure}.
\end{theorem}

\section{On classification of irreducible $\WW_{2,p}$-modules, $p \geq 5$}

\label{klasifikacija-gen}
In this section we briefly discuss the classification of irreducible modules for the vertex operator algebra  $\WW_{2,p}$. One can apply the similar methods as in Section
\ref{klasifikacija-p3}. The main technical problem is  in the determination of the polynomial $\widetilde{G_p}(t)$ from Proposition \ref{evaluation}.

We  have the following conjecture about polynomial $\widetilde{G_p}(t)$:

\begin{conjecture} \label{conj-g}
$$\widetilde{G_p}(t) = B_p { t \choose 3p-1} {t+ p/2 \choose 3p-1} { t+p \choose 3p-1} {t + 3 p /2 \choose 3p-1} { t +2p \choose 3p-1} \quad (B_p \ne 0).$$
\end{conjecture}

 \begin{remark}
We verified   Conjecture \ref{conj-g} for $p=3$ and $p=5$ by using Mathematica package.
Since polynomial $\widetilde{G_p}(t)$ is important in the representation theory of the vertex algebra $\WW_{2,p}$ we plan to return to this Conjecture in our future
work.
\end{remark}

We shall now assume that Conjecture \ref{conj-g} holds. Since $v = G ^2 Q e ^{-5 \a} \in O( \WW_{2,p})$ and   $o(v) v_{\lambda } =\widetilde{G_p}(t) v_{\lambda}$,  we get
the following important relation in Zhu's algebra
$A( \WW_{2,p})$:
\begin{theorem} \label{klasif-general} Assume  that Conjecture \ref{conj-g} holds. Then in Zhu's algebra $A( \WW_{2,p})$ we have:

$$f_p([\omega] ) = 0,$$
where
\bea f_p(x) &=&\left( \prod _{i =1} ^{3p-1} (x - h_{1,i} )  \right) \left(  \prod_{i =1} ^{
 \tfrac{3p-1}{2}
} (x - h_{1,2p-i} )\right) \left(\prod_{i =1} ^{3p-1} (x - h_{2,i} )\right )  \nonumber \\
&=& \left(\prod_{i =1} ^{ \tfrac{p-1}{2}} (x- h_{1,i}) \right)^3  \left(\prod_{i = p } ^{ 2p -1} (x- h_{1,i}) \right)^2 \left(\prod_{i =1} ^{p-1} (x - h_{2,i} )\right )
^2   \nonumber \\
&& \cdot (x-h_{p,2}) \left(\prod_{i=2p} ^{ 3 p -1} (x- h_{1,i}) \right) \left(\prod_{i =2 p } ^{3 p-1} (x - h_{2,i} )\right ).  \eea
\end{theorem}

\vskip 5mm

By using same methods as in Section \ref{klasifikacija-p3}, one can conclude   that  $\WW_{2,p}$ has $ 4 p + \frac{p-1}{2}$ irreducible modules.  The list of irreducible
modules includes $\frac{p-1}{2}$ minimal models for the Virasoro algebra: $$L  (c_{2,p}, h_{1,i}), \quad (1 \le i \le \tfrac{p-1}{2}),$$  and $4 p$ modules $W_p(h)$,
parameterized by  lowest weights
 $$h \in \{ h_{1,i}, h_{2,j}, h_{2,k} \ \vert \ \ p \le i \le 3p-1, \ 1 \le j \le p, \ 2p \le k \le 3p-1  \}. $$

\begin{remark} The relation $f_p([\omega])=0$ also indicates on the   existence of $ {\N}$--graded logarithmic $\WW_{2,p}$--modules whose top components
 have generalized conformal weights $h_{1,i}$  ( $1 \le i \le 2p$) or $h_{2,i}$ ( $1 \le i \le p-1$) with Jordan blocks, with respect to $L(0)$, of size two or three.
 Some  of these logarithmic modules for $\WW_{2,p}$ were constructed explicitly in \cite{AdM-2009}.
\end{remark}

Let $k \geq 0$ and also, let $p \geq 3$. Consider
$$o(G^{k}Qe^{-2k-1}) v_{\lambda}=(-1)^k E_{k,p}(t) v_{\lambda}.$$
Then
$$E_{k,p}(t)
={\rm Res}_{x_1,...,x_{2k+1}} \frac{\Delta(x_1,...,x_{2k+1})^p}{
(x_1 \cdots x_{2k+1})^{(2k+1)p} } \prod_{i=1}^k {\rm ln}
(1-x_{2i}/x_{2i-1})
 \prod_{i=1}^{2k+1} (1+x_i)^t.$$

\begin{conjecture} Let $k \geq 0$, and let also $p \geq 1$ be odd. Then
$${\rm Res}_{x_1,...,x_{2k+1}} \frac{\Delta(x_1,...,x_{2k+1})^p}{ (x_1 \cdots
x_{2k+1})^{(2k+1)p} } \prod_{i=1}^k {\rm ln} (1-x_{2i}/x_{2i-1})
\prod_{i=1}^{2k+1} (1+x_i)^t$$
$$=\lambda_{k,p} \prod_{i=0}^{2k} {t+ \frac{pi}{2} \choose
(k+1)p-1},$$ where $\lambda_{k,p} \neq 0$ is a constant not
depending on $t$.
\end{conjecture}

For $k=0$, we clearly have
$$o(Qe^{-\alpha})v_{\lambda}= {t \choose p-1}v_{\lambda}.$$
The $k=1$ case of the conjecture was verified in Theorem
\ref{conjm1} below. The $k=2$ case is equivalent to Conjecture
\ref{conj-g}

\section{Appendix: Some constant term identities}

In this part we prove an interesting (constant term) identity by
combining combinatorial and representation theoretic methods. We
often provide two different proofs for the same result to make this
part accessible both to combinatorists and algebraists.

\begin{theorem} \label{conjm1} Let $p \geq 1$ be odd. Then
$${\rm Res}_{x_1,x_2,x_3} \frac{1}{(x_1 x_2 x_3)^{3p}} \  {\rm ln}\left(1-\frac{x_2}{x_1}\right) (x_1-x_2)^p (x_1-x_3)^p(x_2-x_3)^p
(1+x_1)^t (1+x_2)^t(1+x_3)^t$$
$$=A_p {t \choose 2p-1} {t+p \choose 2p-1} {t + p/2 \choose 2p-1},$$
where
$$A_p= \frac{(-1)^{(p-1)/2}}{6} \frac{(3p)! (\frac{p-1}{2})!^3 }{ {3p-1 \choose 2p-1} {\frac{5}{2}p-1 \choose 2p-1} (p!)^3 (\frac{3p-1}{2})! }.$$
\end{theorem}

We denote the triple residue in the theorem by $F(p,t)$ and the product of three binomial coefficients
$f(p,t)$ (we omit the constant $A_p$ at this point). It is easy to see that $F$ and $f$ are polynomials of degree at most $6p-3$, by expanding  in Laurent series. So it
is  sufficient to show that these polynomials agree at $6p-2$ values (including derivatives).  More precisely, we will show the following:

\bea
&& F(p,t)=0, \ \ -p \leq t \leq 2p-2, \\
&& F'(p,t)=0, \ \  0 \leq t \leq p-2 \\
&& F(p,t)=0, \ \   \frac{-p}{2} \leq t \leq \frac{3p}{2}-2  \ \ {\rm and} \\
&& F(p,2p-1)=f(p,2p-1) \neq 0.
\eea
All these properties can be easily verified by $f(p,t)$.

We start with a combinatorial lemma which also has a representation theoretic version (see Proposition \ref{app-prop1} below).
We include the proof here because it is needed in Lemma \ref{app-lemma3} below.

\begin{lemma} \label{app-lemma1} We have $F(p,t)=0$ for $ 0 \leq t \leq 2p-2$.
\end{lemma}
{\em Proof.} We expand $F(p,t)$ by using binomial expansion. This leads to summation
over $7$ variables. By using residue condition we trim down to four summation variables:
\bea \label{binomials}
&&  \ \ \ \ \ \ F(p,t) \\
&& =\sum_{i,j,k,m} (-1)^{i+j+k} \frac{1}{m} {p \choose i}{p \choose j}{p \choose k}{t \choose m+3p-1-i-j}{t \choose 2p-m-1+i-k}{t \choose p-1+j+k} . \nonumber
\eea

Observe that ${m \choose n}=0$, for $n <m$. We argue that for $0 \leq t \leq 2p-2$ at least one of $t$-binomial coefficients is zero. If not then then  $m+3p-1-i-j \leq
2p-2$, $2p-m-1+i-k \leq 2p-2$ and
$p-1+j+k \leq 2p-2$. But we get contradiction because $(m+3p-1-i-j)+(2p-m-1+i-k)+(p-1+j+k)=6p-3$.

\begin{lemma} We have $F'(p,t)=0$ for $t=0,..,p-2$.

\end{lemma}

{\em Proof.} If we differentiate $F(p,t)$ with respect to the $t$ variable we obtain
$$F'(p,t)=F(p,t)({\rm ln}(1+x_1)+{\rm ln}(1+x_2)+{\rm ln}(1+x_3)).$$
Now,  for $t=0, \dots, p-2$ we clearly have
$$Res_{x_3} F(p,t)({\rm ln}(1+x_1)+{\rm ln}(1+x_2))=0.$$
But we also have
$${\rm Res}_{x_1}({\rm ln}(1+x_3)F(p,t)=0.$$
The proof follows.

\qed

It is interesting that Lemma \ref{app-lemma1} has a representation theoretic version.
We include it here as an illustration of the method.

\begin{proposition} \label{app-prop1} Recall $H(0) v_{\lambda}=F(p,t) v_{\lambda}$, where $ \langle \lambda,\alpha \rangle =t$.
Then the polynomial $F(p,t)$ is divisible with $t(t-1) \cdots (t-2p+2)$.
\end{proposition}

{\em Proof.} We recall  $H=GF=GQe^{-3 \alpha}$. It is sufficient to show that

$$H(0)e^{\frac{i \alpha}{2p}}=H_{6p-4} e^{\frac{i \alpha }{2p }}=0$$  for $i=0,2,...,{4p-4}.$

It is easy to see that

$$Q e^{\frac{i \alpha }{2p }}=G e^{\frac{i \alpha}{2p }}=0$$

for all $i$ as above.

Also, observe the relation

$$(GF)_{6p-4}=G F_{6p-4}-F_{6p-4} G$$
and
$$(Qe^{-3 \alpha})_{6p-4}=Qe^{-3 \alpha}_{6p-4}-e^{-3 \alpha}_{6p-4} Q.$$

Therefore

$$H(0)e^{\frac{i \alpha }{2p }}=G Q e^{-3 \alpha}_{6p-4} e^{\frac{i  \alpha}{2p}}.$$

From

$$e^{-3 \alpha}_{6p-4} e^{\frac{i  \alpha}{2p}}={\rm Coeff}_{x^{-6p+3}} Y(e^{-3 \alpha},x)e^{\frac{i  \alpha}{2p}}.$$

and

$$ \langle -3 \alpha, \frac{i \alpha}{2 p} \rangle =-\frac{3}{2} i.$$

it follows that $$e^{-3 \alpha}_{6p-4} e^{\frac{i  \alpha}{2p}}=0$$

for all $i$ even from $0$ to $4p-4$.

\qed

\begin{proposition} \label{contragred}
 The polynomial $F$ satisfies $F(p,t)=-F(p, p-t-2)$.  Consequently,  $F(p,t)=0$, $-p \leq t \leq -1$.
\end{proposition}

\begin{proof} Denote by $M(1,\lambda)^\circ$ the contragradient module of
$M(1,\lambda)$.

Then  $$ \langle H(0)w',w \rangle=- \langle w',H(0)w \rangle.$$ But
$M(1,\lambda)^\circ \cong M(1,p-\lambda-2)$ (see \cite{AdM-2007}),
so the proof  follows. The second statement follows now form Lemma
\ref{app-lemma1}.
\end{proof}

For half-integer roots more we shall make use of  twisted $V_L$-modules and operator $G^{tw}$ introduced earlier.

\begin{lemma}  \label{tw-1} We have
$$G^{tw} e^{\frac{2\ell +1}{2p} \alpha}=0, \ \ {\rm for} \ \ \ell  \geq -\frac{p+1}{2}.$$
\end{lemma}
\begin{proof} The proof follows from the fact
$$wt(e^{\frac{2\ell+1}{2p} \alpha})<wt(e^{2 \alpha +\frac{2\ell+1}{2p} \alpha}),$$
for $\ell $ in this range \end{proof}

\begin{lemma} \label{tw-2}
Assume that $\ell \le \frac{3p-1}{2}$, then
$$F(0) e^{\frac{2\ell +1}{2p} \alpha}=0.$$
\end{lemma}
\begin{proof}
By using the definition of twisted $V_L$--modules, we have that
the vector
$$F(0) e^{\frac{2\ell +1}{2p} \alpha}$$
can be expressed as a linear combination of vectors of the form
\bea &&(e^{\a}_n e ^{-3 \a} )_{6p-4} e ^{\tfrac{\ell}{p} \a}, \quad
(n \in {  \mathbb{Z}} ). \label{lin-comb} \eea

But every vector (\ref{lin-comb}) belongs to $M(1) \otimes e
^{\tfrac{\ell}{p} \a - 2 \a}$. Since
$$wt(e ^{\tfrac{\ell}{p} \a})<wt(e
^{\tfrac{\ell}{p} \a - 2 \a}) \quad \mbox{if} \ \ell \le
\frac{3p-3}{2},$$ we get the assertion.
\end{proof}

\begin{proposition} $F(p,t)=0$ for $t \in \{ -\frac{p}{2}, - \frac{p}{2} + 1,\dots,- \frac{p}{2}+2p-2 \}$.
\end{proposition}
\begin{proof}
The proof follows from Lemma \ref{tw-1}, Lemma \ref{tw-2} and
relation
$$ H(0) e^{\frac{2\ell +1}{2p} \alpha} =  G ^{tw}F(0)e^{\frac{2\ell +1}{2p} \alpha}-F(0) G^{tw} e^{\frac{2\ell +1}{2p}
\alpha}.$$
\end{proof}

\begin{lemma} \label{app-lemma3} We have

$$F(p,2p-1) =  \sum_{m=1}^p  \sum_{k = 0}^p  \frac{(-1)^{m+k}}{m} { p \choose k}^2 {p \choose m+k}.$$
\end{lemma}

\begin{proof} The main point for choosing $t=2p-1$ comes from equation
(\ref{binomials}). Observe that three $t$-binomial coefficients in
(\ref{binomials}) will be nonzero only if $i,j,k, m$ satisfy

$$(m+3p-1-i-j)=2p-1,$$ $$(2p-m-1+i-k)=2p-1,$$ and $$(p-1+j+k)=2p-1.$$
This gives $i=m+k$ and $j=p-k$ where $k$ is the free variable. Therefore
$$F(p,2p-1)=\sum_{m = 1}^p \sum_{k =0}^p \frac{(-1)^{m+k}}{m} { p \choose k}^2 {p \choose m+k}.$$
\end{proof}

The next result will give exact evaluation for the sum $F(2,2p-1)$. Similar formulas were obtained
earlier

\begin{proposition}
We have
$$\sum_{m = 1}^p \sum_{k = 0}^p  \frac{(-1)^{m+k}}{m} { p \choose k}^2 {p \choose m+k} =\frac{(-1)^{(p+1)/2}}{3} \frac{(3p)!!}{p!!^3}.$$
%
\end{proposition}
{\em Proof.}
We define the $n$-th harmonic number
$$H_n=\sum_{i=1}^n \frac{1}{i},$$
and consider the sums
$$C_{k}= \sum_{m=1}^p (-1)^m \frac{1}{m} {p  \choose m+k}.$$

It is not hard to see (simply by using basic properties of binomial coefficients) that
$$C_{0}=-H_p.$$
Similarly, by using induction and the formula ${p \choose k}={p-1 \choose k}+{p-1 \choose k-1}$, we prove
the relation
$$C_{k}=-{p \choose k}(H_n-H_k), \ \ k \geq 1.$$
Putting together we obtain
\bea
&& \sum_{m = 1}^p \sum_{k = 0}^p  \frac{(-1)^{m+k}}{m} { p \choose k}^2 {p \choose m+k} \nonumber \\
&& = - \sum_{k=0}^p (-1)^k {p \choose k}^3(H_p-H_k)=0+\sum_{k=0}^p (-1)^k {p \choose k}^3 H_k \nonumber
\eea

Fortunately the last sum is a difficult, albeit known, binomial-harmonic sum computed by W. Chu and M. Fu (cf. Example 1
on page 3 in \cite{CF}). Their identity reads
$$\sum_{k=0}^{2m+1}  (-1)^k {2m+1 \choose k}^3 H_k=\frac{(-1)^{m+1}(6m+3)! (m!)^3}{6 \cdot (1+3m)!(1+2m)!^3}.$$

The rest follows easily if we let $p=2m+1$ and observe an easy identity

$$(-1)^{(p+1)/2} \frac{(3p)!!}{3(p)!!^3}=\frac{(-1)^{m+1}(6m+3)! (m!)^3}{6 \cdot (1+3m)!(1+2m)!^3}.$$

\qed

The previous identity can be also reformulated as a constant term
identity similar  with three variables, but now "disturbed"
logarithmically!

\begin{corollary} \label{log-dyson3} Let $p \geq 1$ be an odd integer. Then

$${\rm CT}_{x_1,x_2,x_3} {\rm ln}\left(1-\frac{x_2}{x_1} \right)\left(1-\frac{x_2}{x_1} \right)^p \left(1- \frac{x_3}{x_2} \right)^p \left(1-\frac{x_1}{x_3}\right)
^p=\frac{(-1)^{(p+1)/2}}{3} \frac{(3p)!!}{p!!^3}.$$
%
\end{corollary}

The previous corollary brings us to the realm of Dyson's contant term identities  \cite{An}, which state that for any positive $k \in \mathbb{N}$
and $n \geq 2$ the constant term
$${\rm CT}_{x_1,...,x_k} \frac{1}{(x_1 \cdots x_{k})^{m(k-1)}}  \prod_{1 \leq i < j \leq k} (x_i-{x_j})^{2m}$$
is equal (up to a sign) to $$\frac{(k m)!}{(m!)^k}.$$
Notice that the power of the Vandermonde factor is always even (if the power is odd the constant term
is trivially zero!). That is why any identity involving $\prod_{i<j} (x_i-x_j)^{2m+1}$  has to include additional factors.

Motivated by Corollary \ref{log-dyson3} we present a very precise conjecture in this direction.
\begin{conjecture} Let $k$ and $m$ be positive integers. Then (up to a sign)
$${\rm CT}_{x_1,...,x_{2k+1}}  \frac{1}{(x_1 \cdots x_{2k+1})^{(2m+1)k}} \prod_{i=1}^k {\rm ln}\left(1-\frac{x_{2i}}{x_{2i-1}}\right)  {\prod_{1 \leq i <j \leq 2k+1
}(x_i-x_j)^{2m+1}}$$ equals
$$\frac{((2k+1)(2m+1))!!}{(2k+1)!!(2m+1)!!^{2k+1}}.$$
\end{conjecture}
We verified this conjecture for small $k$ and $m$ by using computer.

\end{document}